\documentclass[12pt]{article}

\usepackage{amsmath}\usepackage{amsfonts}\usepackage{amscd}
\usepackage{amsthm}

\newtheorem{theorem}{Theorem}
\newtheorem{lemma}{Lemma}[section]
\newtheorem{corollary}{Corollary}[section]

\newtheorem{remark}{Remark}[section]

\newcommand{\Ker}{\mathop{\mathrm{Ker}}}
\newcommand{\Pol}{\mathop{\mathrm{Pol}}}
\newcommand{\SQ}{\mathrm{SQ}}
\newcommand{\N}{\mathbb{N}}
\newcommand{\R}{\mathbb{R}}
\newcommand{\Z}{\mathbb{Z}}
\newcommand{\D}{\Delta}
\renewcommand{\O}{\mathrm{O}}
\renewcommand{\o}{\mathrm{o}}

\title{Asymptotically polynomial solutions \\
of difference equations of neutral type}
\author{Janusz Migda\\ \\
 Faculty of Mathematics and Computer Science,\\ A. Mickiewicz University,
Umultowska 87, 61-614 Pozna\a'n, Poland;\\
email: migda@amu.edu.pl} 

\date{}
\begin{document}
\maketitle

\begin{abstract}
Asymptotic properties of solutions of difference equation of the form
\[
\Delta^m(x_n+u_nx_{n+k})=a_nf(n,x_{\sigma(n)})+b_n
\]
are studied. We give sufficient conditions under which all solutions, 
or all solutions with polynomial growth, or all nonoscillatory solutions are 
asymptotically polynomial. We use a new technique which allows us to control 
the degree of approximation. 
\smallskip\\
{\bf Key words:} difference equation, neutral equation, asymptotic behavior, 
asymptotically polynomial solution, nonoscillatory solution.
\smallskip\\
{\bf AMS Subject Classification:} $39A10$
\end{abstract}

\section{Introduction}

Let \ $\N$, $\Z$, $\R$ \ denote the set of positive integers, all 
integers and real numbers respectively. 
Let \ $m\in\N$, \ $k\in\Z$. \ 
We consider asymptotic properties of solutions of difference 
equations of the form 
\begin{equation}\label{E}\tag{E}
\Delta^m(x_n+u_nx_{n+k})=a_nf(n,x_{\sigma(n)})+b_n
\end{equation}
\[
u_n,a_n,b_n\in\R, \quad f:\N\times\R\to\R, \quad 
\sigma:\N\to\Z, \quad \sigma(n)\to\infty, \quad u_n\to c\in\R, \quad |c|\neq 1.
\]
By a solution of (\ref{E}) we mean a sequence \ $x:\N\to\R$ \ satisfying (\ref{E}) 
for all large $n$. \\ 
Asymptotic properties of solutions of neutral difference equations were investigated 
by many authors. These studies tend in several directions. For example, the papers 
\cite{Cheng 1998}, \cite{MJ Migda 2004}, \cite{MJ Migda 2009} and \cite{Cheng 2006} 
are devoted to the classification of solutions. In  \cite{Huang}, \cite{Kang 2009}, 
\cite{Liu 2010} and \cite{Zhou} where studied solutions with prescribed asymptotic 
behavior. In \cite{Agarwal 1996}, \cite{Bolat 2004}, \cite{Karpuz 2009}, 
\cite{Thandapani 1997} were investigated oscillatory solutions. Asymptotically 
polynomial solutions were studied in \cite{MJ Migda 2005}, \cite{M Migda 2006}, 
\cite{Thandapani 2004}, \cite{Wang Sun}. Asymptotically polynomial solutions 
were also studied in continuous case, see for example \cite{Dzurina 2002}, 
\cite{Rogovchenko 2008}, \cite{Naito 1998}.\\ 
Thandapani, Arul and Raja in \cite{Thandapani 2004} establish conditions under which 
for any nonoscillatory solution \ $x$ \ of the equation 
\begin{equation}\label{E2}
\D^2(x_n+px_{n+k})=f(n,x_{n+l})
\end{equation}
there exists a constant \ $a$ \ such that 
\[
x_n=an+\o(n).
\]
In \cite{MJ Migda 2005}, there are given conditions under which any nonoscillatory solution \ 
$x$ \ of \eqref{E2} has an asymptotic behavior 
\[
x_n=an+b+\o(1).
\]
M. Migda, in \cite{M Migda 2006}, establish conditions under which for any nonoscillatory  solution \ $x$ \ of \eqref{E} there exists a constant \ $a$ \ such that  
\[
x_n=an^{m-1}+\o(n^{m-1}).
\]
In this paper, in Theorem \ref{T1}, we extend these results in the following way. 
Let \ $s\in(-\infty,m-1]$ \ and let \ $p$ \ be a nonnegative integer such that \ 
$s\leq p\leq m-1$. \ We establish conditions under which any solution, or any   
solution with polynomial growth, or any nonoscillatory solution \ $x$ \ has an 
asymptotic behavior 
\[
x_n=a_{m-1}n^{m-1}+a_{m-2}n^{m-2}+\dots+a_pn^p+\o(n^s)
\]
for some fixed real \ $a_{m-1},a_{m-2},\dots,a_p$. 
\smallskip\\
The idea of the proof is as follows. Let \ $z$ \ be a sequence defined by 
\begin{equation}\label{z1}
z_n=x_n+u_nx_{n+k}.
\end{equation} 
Using \ $z$ \ we can write equation \eqref{E} in the form 
\begin{equation}\label{Ez}
\Delta^mz_n=a_nf(n,x_{\sigma(n)})+b_n.
\end{equation} 
Let \ $s$ \ be a real number such that \ $s\leq m-1$. \ Assume that 
\[
\sum_{n=1}^\infty n^{m-1-s}|a_n|<\infty \qquad \text{and} \qquad 
\sum_{n=1}^\infty n^{m-1-s}|b_n|<\infty.
\]
Using a Bihari type lemma and some additional assumptions, we show that \eqref{Ez} 
implies 
\begin{equation}\label{z2}
\sum_{n=1}^\infty n^{m-1-s}|\D^mz_n|<\infty.
\end{equation}
Next we use the result from \cite{J Migda 2013}, which states that if \ $\D^mz$ \ 
is asymptotically zero, then \ $z$ \ is asymptotically polynomial. More precisely, 
we show that \eqref{z2} implies 
\begin{equation}\label{z3}
z_n=\varphi(n)+\o(n^s)
\end{equation}
where \ $\varphi$ \ is a polynomial sequence such that \ $\deg\varphi<m$. \ 
Finally, using our Lemma \ref{ASL5}, we show that 
\begin{equation}\label{x1}
x_n=\psi(n)+\o(n^s)
\end{equation}
for certain polynomial sequence \ $\psi$ \ such that \ $\deg\psi<m$. \ 
In the last section we show, that if \ $s=q$ \ is a nonnegative integer, \ 
then \eqref{x1} may be replaced by a stronger condition  
\[
x_n=\psi(n)+w_n, \qquad 
\D^kw_n=\o(n^{q-k}) \quad \text{for} \quad k=0,1,\dots,q.
\]
The paper is organized as follows. In Section 2, we introduce notation and 
terminology. Section 3 is devoted to the proof of Lemma \ref{ASL5}. In Section 4,  
we obtain Theorem \ref{T1}, which is the main result of this paper. The proof of 
Theorem \ref{T1} is based on three lemmas: Lemma \ref{ASL5}, Lemma \ref{BHL}, 
and Lemma \ref{SDL}. In Section 5, we obtain a result analogous to Theorem \ref{T1}, 
but we replace the spaces of asymptotically polynomial sequences by the spaces of 
regularly asymptotically polynomial sequences (see \eqref{raspol}).

\section{Notation and terminology}

By \ $\SQ$ \ we denote the space of all sequences \ $x:\N\to\R$. \ 
If \ $p,k\in\Z$, \ $k\geq p$ \ then 
\[
\N(p,k)=\{p,p+1,\dots,k\}, \qquad \N(p)=\{p,p+1,\dots \}.
\]
For \ $m\in\N(0)$, \ we define 
\[
\Pol(m-1)=\Ker\D^m=\{x\in\SQ:\D^mx=0\}.
\]
Then \ $\Pol(m-1)$ \ is the space of all polynomial sequences of degree less 
than \ $m$. \ Note that 
\[
\Pol(-1)=\Ker\D^0=0
\]
is the zero space. For \ $x,y\in \SQ$, \ we define the product \ $xy$ \ by \  $(xy)(n)=x_ny_n$ \ for any \ $n$. \ Moreover, \ $|x|$ \ denotes the sequence defined 
by \ $|x|(n)=|x_n|$ \ for any $n$. 
\medskip\\
We use the symbols \ ''big $\O$'' \ and \ ''small $\o$'' \ in the usual sense but 
for \ $a\in\SQ$ \ we also regard \ $\o(a)$ \ and \ $\O(a)$ \ as subspaces 
of \ $\SQ$. \ More precisely 
\[
\o(1)=\{x\in\SQ: x_n\to 0\}, \qquad \O(1)=\{x\in\SQ: x \ \text{is bounded}\}
\]
\[
\o(a)=a\o(1)=\{ax: x\in\o(1)\}, \qquad  \O(a)=a\O(1)=\{ax: x\in\O(1)\}.
\]
For a subset \ $X$ \ of \ $\SQ$, \ let 
\[
\D^mX=\{\D^mx: x\in X\}, \qquad \D^{-m}X=\{z\in\SQ: \D^mz\in X\}
\]
denote respectively the image and the inverse image of \ $X$ \ under the map \ 
$\D^m:\SQ\to\SQ$. \ 
Now, we can define spaces of asymptotically polynomial sequences and regularly  asymptotically polynomial sequences 
\begin{equation}\label{raspol}
\Pol(m-1)+\o(n^s), \qquad\quad \Pol(m-1)+\D^{-k}\o(1),
\end{equation}
where \ $s\in(-\infty,m-1]$ \ and \ $k\in\N(0,m-1)$. \ Moreover, let 
\[
\o(n^{-\infty})=\bigcap_{s\in\R}\o(n^s)=\bigcap_{k=1}^\infty\o(n^{-k}), \qquad 
\O(n^\infty)=\bigcup_{s\in\R}\O(n^s)=\bigcup_{k=1}^\infty\O(n^k).
\]
Note that the condition \ $\limsup\sqrt[n]{|a_n|}<1$ \ or \  $\lim\sup\dfrac{|a_{n+1}|}{|a_n|}<1$ \ implies \ $a\in\o(n^{-\infty})$.
\smallskip\\
Let \ $x,u\in\SQ$ \ and \ $k\in\Z$. \ We say that \ $x$ \ is \ nonoscillatory \ 
if \ $x_nx_{n+1}\geq 0$ \ for large \ $n$. \ If \ $x_nx_{n+k}\geq 0$ \ for 
large \ $n$ \ we say that \ $x$ \ is \ $k$-nonoscillatory. If \ $x_nu_nx_{n+k}\geq 0$ \ 
for large \ $n$ \ we say that \ $x$ \ is \ $(u,k)$-nonoscillatory. 
\begin{remark}
If \ $\liminf u_n>0$, \ then a sequence \ $x$ \ is $(u,k)$-nonoscillatory 
if and only if it is $k$-nonoscillatory. If \ $\limsup u_n<0$, \ then \ 
$x\in\SQ$ \ is $(u,k)$-nonoscillatory if and only if \ $-x$ \ is $k$-nonoscillatory. 
Every nonoscillatory sequence \ $x$ \ is also $k$-nonoscillatory for any \ $k\in\Z$. 
\end{remark}
Let \ $X$ \ be a metric space. A function $g:X\to\mathbb{R}$ is called locally 
bounded if for any \ $x\in X$ \ there exists a neighborhood \ $U$ \ of \ $x$ \ such 
that the restriction \ $g|U$\ is bounded. 
\begin{remark}
If \ $X$ \ is a closed subset of \ $\mathbb{R}$, \ then a function \ 
$g:X\to\mathbb{R}$ \ is locally bounded if and only if it is bounded on every 
bounded subset of \ $X$. \ On the other hand if, for example, \ $h:(0,\infty)\to\R$ \ 
is given by \ $g(t)=t^{-1}$, \ then \ $g$ \ is locally bounded and \ $g|(0,1)$ \ is  unbounded.
\end{remark}
Let \ $f:\N\times\R\to\R$, \ $g:[0,\infty)\to[0,\infty)$, \ and $p\in\R$. \ 
We say that \ $f$ \ is \ $(g,p)$-bounded if 
\[
|f(n,t)|\leq g\left(\frac{|t|}{n^p}\right)
\]
for any \ $(n,t)\in\N\times\R$. 
\smallskip\\
We say that a sequence \ $x$ \ is of polynomial growth if \ $x\in\O(n^\infty)$.

\section{Associated sequences} 

In this section we assume that \ $x,u,z\in\SQ$, \ $k\in\Z$, \ $\lim u_n=c\in\R$, \ 
$|c|\neq 1$ \ and 
\[
z_n=x_n+u_nx_{n+k}, \qquad \text{for} \qquad n\geq\max(0,-k).
\]
This section is devoted to the proof of Lemma \ref{ASL5}. In this lemma, we establish  conditions under which, for a given real \ $\alpha$, \ the condition \  $z\in\Pol(m)+\o(n^\alpha)$ \ implies 
\[
x\in\Pol(m)+\o(n^\alpha). 
\]
Lemma \ref{ASL5} extends \cite[Lemma 4]{MJ Migda 2005} 
and will be used in the proof of Theorem \ref{T1}.

\begin{lemma}\label{ASL1}
Assume \ $x$ \ is bounded and \ $z$ \ is convergent. Then \ $x$ \ is convergent and 
\[ 
(1+c)\lim_{n\to\infty}x_n=\lim_{n\to\infty}z_n.
\]
\end{lemma}
\begin{proof}
See Lemma 1 in \cite{MJ Migda 2005}.
\end{proof}

\begin{remark}
Boundedness of \ $x$ \ cannot be omitted in Lemma \ref{ASL1}. For example, 
if \ $x_n=2^n$, \ $u_n=-2^{-1}$ \ and \ $k=1$, \ then \ $z_n=0$ \ for any \ $n$ \ 
and \ $x$ \ is divergent. However, see the next lemma.
\end{remark}
\begin{lemma}\label{ASL2}
Assume one of the following conditions is satisfied
\[
(a) \quad |c|<1 \quad \text{and} \quad k\leq 0, \qquad\quad 
(b) \quad |c|>1 \quad \text{and} \quad k\geq 0.
\]
Then boundedness of the sequence \ $z$ \ implies boundedness of \ $x$.
\end{lemma}
\begin{proof}
Assume  \ $(a)$ \ and the sequence \ $z$ \ is bounded. Choose \ $b>0$ \ 
such that \ $|z_n|\leq b$ \ for all \ $n$. \ Choose a number \ $\beta$ \ 
such that \ $|c|<\beta<1$. \ Let \ $r=-k$. \ Then \ $r\geq 0$ \ and there 
exists \ $n_0\geq r$ \ such that \ $|u_n|<\beta$ \ for \ $n\geq n_0$. \ Let 
\[ 
K=\max(|x_0|,\dots,|x_{n_0}|), \qquad n\in\N(n_0). 
\]
There exists \ $m\in\N(0)$ \ such that 
\[
0\leq n-mr\leq n_0,\quad n-(m-1)r>n_0.
\]
Since \ $x_n=z_n-u_nx_{n-r}$, \ we obtain 
\[
|x_n|\leq b+|u_n||x_{n-r}|<b+\beta|x_{n-r}|. 
\]
Similarly \ $|x_{n-r}|<b+\beta|x_{n-2r}|$. \ Hence 
\[ 
|x_n|<b+\beta b+\beta^2|x_{n-2r}| 
\]
and so on. After $m$ steps we obtain  
\[
|x_n|<b(1+\beta+\beta^2+\dots+\beta^{m-1})+\beta^m|x_{n-mr}|.
\]
Since \ $\beta\in \ (0,1)$ \ and \ $n-mr\leq n_0$, \ we have \ 
$\beta^m|x_{n-mr}|<K$. \ Hence 
\[
|x_n|<\frac{b}{1-\beta}+K.
\]
So, the sequence \ $(x_n)$ \ is bounded. 
\smallskip\\
Now, assume \ $(b)$. \ Let
\[
v_n=\frac{1}{u_n}, \qquad  c'=\frac{1}{c}, \qquad y_n=u_nx_{n+k}.
\]
Then \ $|c'|<1$, \ $\lim v_n=c'$ \ and \ $y_n+v_ny_{n-k}=u_nx_{n+k}+x_n=z_n$. \ 
Hence, by first part of the proof, the sequence \ $y$ \ is bounded. Therefore 
the sequence  \ $x=z-y$ \ is bounded too. The proof is complete. 
\end{proof}
Lemma \ref{ASL2} extends \cite[Lemma 2]{MJ Migda 2005}.

\begin{lemma}\label{ASL3}
If \ $x\in \O(n^\infty)$ \ and \ $z$ \ is bounded, then \ $x$ \ is bounded. 
\end{lemma}
\begin{proof}
By Lemma \ref{ASL2}, we can assume that one of the following conditions is satisfied
\[
(a) \quad |c|<1 \quad \text{and} \quad k>0, \qquad\quad 
(b) \quad |c|>1 \quad \text{and} \quad k<0.
\]
Assume (a) and choose \ $M>1$ \ such that \ $|z_n|\leq M$ \ for all \ $n$. \ 
Choose a number \ $\rho$ \ such that \ $|c|<\rho<1$. \ There exists an index \ $n_1$ \ 
such that \ $|u_n|<\rho$ \ for \ $n\geq n_1$. \ Then 
\begin{equation}\label{rho}
|z_n-x_n|=|u_n||x_{n+k}|<\rho|x_{n+k}|
\end{equation}
for \ $n\geq n_1$. \ Let \ $r=\rho^{-1}$. \ Then \ $r>1$ \ and, by \eqref{rho},  
\[
|x_{n+k}|>r|z_n-x_n| 
\] 
for \ $n\geq n_1$. \ Choose a constant \ $N$ \ such that 
\begin{equation}\label{N}
N>\frac{1}{r-1}. 
\end{equation}
Assume the sequence \ $(x_n)$ \ is unbounded. Then there exists \ 
$p\geq n_1$ \ such that 
\begin{equation}\label{N+1} 
|x_{p}|\geq(N+1)M. 
\end{equation}
Since \ $|z_{p}|\leq M$, \ by \eqref{N+1}, we have \ $|z_{p}-x_{p}|\geq NM$. \ Then 
\[ 
|x_{p+k}|>r|z_{p}-x_{p}|\geq rNM. 
\] 
The condition \ $|z_{p+k}|\leq M$ \ implies 
\[ 
|z_{p+k}-x_{p+k}|\geq rNM-M=(rN-1)M. 
\]
Hence 
\begin{equation}\label{p+2k}
|x_{p+2k}|>r|z_{p+k}-x_{p+k}|\geq r(rN-1)M. 
\end{equation}
Since \ $|z_{p+2k}|\leq M$, \ we obtain 
\[ 
|z_{p+2k}-x_{p+2k}|\geq r(rN-1)M-M=(r(rN-1)-1)M.
\]
If 
\[ 
a_1=rN, \quad a_2=r(a_1-1), \ \dots, \ a_{n+1}=r(a_n-1) ,
\] 
then, as in \eqref{p+2k}, we have 
\begin{equation}\label{p+nk}
|x_{p+nk}|\geq a_nM 
\end{equation}
for \ $n\geq 1$. \ Moreover, 
\[ 
a_2=r(a_1-1)=r^2N-r, \qquad a_3=r(a_2-1)=r^3N-r^2-r
\]
and so on. Hence, for \ $n\geq 1$, \ we obtain 
\[
a_n=r^nN-(r^{n-1}+r^{n-2}+\dots+r+1)+1
\]
\[
=r^nN-\frac{r^n-1}{r-1}+1=
\left(N-\frac{1}{r-1}\right)r^n+\frac{1}{r-1}+1.
\]
Let 
\[
a=N-\frac{1}{r-1}, \qquad b=\frac{1}{r-1}+1.
\]  
By \eqref{N}, \ $a>0$. Since \ $r>1$, \ we have \ $b>0$. \ Moreover, by \eqref{p+nk}, 
\[
|x_{p+nk}|\geq ar^n+b 
\]
for \ $n\geq 1$. \ Since \ $x\in \O(n^\infty)$, \ there exists a number \ $\alpha>1$ \ 
such that \ $x_n=\O(n^\alpha)$. \ There exist \ $w\in(0,\infty)$ \ and \ 
$m_0\in\N(0)$ \ such that 
\[
(p+nk)^\alpha<wn^\alpha
\]
for \ $n\geq m_0$. \ Then  
\[
\frac{x_{p+nk}}{(p+nk)^\alpha}>\frac{ar^n+b}{(p+nk)^\alpha}>
\frac{a}{w}\frac{r^n}{n^\alpha}
\]
for \ $n\geq m_0$. \ It is impossible since \ $r>1$ \ and \ $x_n=\O(n^\alpha)$. \ 
Hence, the sequence \ $(x_n)$ \ is bounded. 
Now assume \ (b) \ and \ $x_n=\O(n^\alpha)$. \ Let 
\[
v_n=\frac{1}{u_n}, \qquad  c'=\frac{1}{c}, \qquad y_n=u_nx_{n+k}.
\]
Then 
\[
|c'|<1, \quad \lim v_n=c', \quad y_n=\O(n^\alpha), \quad  
y_n+v_ny_{n-k}=u_nx_{n+k}+x_n=z_n 
\]
and by the first part of the proof the sequence \ $(y_n)$ \ is bounded. Hence, the 
sequence \ $x_n=z_n-y_n$ \ is bounded too. 
\end{proof}

\begin{lemma}\label{ASL4}
Let \ $\alpha\in\R$. \ Assume \ $k(|c|-1)\geq 0$ \ or \ $x\in \O(n^\infty)$. \ 
Then 
\begin{itemize}
\item[$(1)$] \ if \ $z_n=\O(n^\alpha)$, \ then \ $x_n=\O(n^\alpha)$,
\item[$(2)$] \ if \ $z_n=\o(n^\alpha)$, \ then \ $x_n=\o(n^\alpha)$.
\end{itemize}
\end{lemma}
\begin{proof}
Assume \ $\alpha=0$. \ If \ $k(|c|-1)\geq 0$, \ then the result follows from 
Lemma \ref{ASL2} and Lemma \ref{ASL1}. If \ $x\in \O(n^\infty)$, \ then by 
Lemma \ref{ASL3} boundedness of \ $z$ \ implies boundedness of \ $x$. \ 
Moreover, by Lemma \ref{ASL3} and Lemma \ref{ASL1}, convergence of \ $z$ \ implies 
convergence of \ $x$. \ Now assume \ $\alpha$ \ is an arbitrary real 
number. By the equality 
\[
\frac{z_n}{n^\alpha}=\frac{x_n}{n^\alpha}+u_n\frac{(n+k)^\alpha}{n^\alpha}
\frac{x_{n+k}}{(n+k)^\alpha}=
\frac{x_n}{n^\alpha}+u_n\left(1+\frac{k}{n}\right)^\alpha\frac{x_{n+k}}{(n+k)^\alpha}
\]
and the equality
\[
\lim_{n\to\infty}\left(1+\frac{k}{n}\right)^\alpha=1
\]
we see that the result is a consequence of the first part of the proof.
\end{proof}
Now, we are ready to state and prove the main result of this section.

\begin{lemma}\label{ASL5}
Assume \smallskip \ $k(|c|-1)\geq 0$ \ or \ $x\in \O(n^\infty)$. \ Let \ 
$m\in\N(0)$, \ $\alpha\in\R$, \ and 
\[
u_n=c+\o(n^{\alpha-m}). 
\]
Then the condition \ $z\in\Pol(m)+\o(n^\alpha)$ \ implies \ $x\in\Pol(m)+\o(n^\alpha)$.
\end{lemma}
\begin{proof}
If \ $\alpha>m$, \ then 
\[
\Pol(m)+\o(n^{\alpha})=\o(n^{\alpha}) 
\]
and the assertion follows from Lemma \ref{ASL4}. Assume \ $\alpha\leq m$. \ 
For \ $n\geq\max(0,-k)$, \ let 
\[
z'_n=x_n+cx_{n+k}. 
\]
We will show, by induction on \ $m$, \ that 
\begin{equation}
z'\in\Pol(m)+\o(n^\alpha) \ \Longrightarrow\  x\in\Pol(m)+\o(n^\alpha).
\end{equation}
For \ $m=-1$ \ this assertion follows from Lemma \ref{ASL4}. Assume it is true 
for certain \ $m\geq -1$ \ and let 
\[
z'\in \Pol(m+1)+\o(n^\alpha). 
\]
Then there exist \ $a\in\R$ \ and \ $w\in\Pol(m)+\o(n^\alpha)$ \ such that 
\[ 
z'_n=an^{m+1}+w_n. 
\]
Since
\[
a=\frac{a}{1+c}+\frac{ca}{1+c},\quad (n+k)^{m+1}=n^{m+1}+r_n,\quad 
r\in \Pol(m)
\]
we obtain 
\[
w_n=z'_n-an^{m+1}=x_n-\frac{a}{1+c}n^{m+1}+cx_{n+k}-\frac{ca}{1+c}n^{m+1}
\]
\[
=\left(x_n-\frac{a}{1+c}n^{m+1}\right)+
c\left(x_{n+k}-\frac{a}{1+c}(n+k)^{m+1}+\frac{a}{1+c}r_n\right).
\]
Let
\[
v_n=x_n-\frac{a}{1+c}n^{m+1}. 
\]
Then
\[ 
w_n-\frac{ca}{1+c}r_n=v_n+cv_{n+k}.
\]
Since \ $r,w\in \Pol(m)+\o(n^\alpha)$ \ we obtain 
\[
\left(w-\frac{ca}{1+c}r\right)\in \Pol(m)+\o(n^\alpha).
\]
The condition \ $z'\in \Pol(m)+\o(n^\alpha)$ \ implies \ $z'=\O(n^m)$. \ Hence, 
by Lemma \ref{ASL4}, \ $x_n=\O(n^m)$. \ Therefore \ $v\in \O(n^\infty)$ \ and, 
by inductive hypothesis, 
\[
v\in \Pol(m)+\o(n^\alpha). 
\]
By the equality 
\[
x_n=v_n+\frac{a}{1+c}n^{m+1},
\]
we have 
\[
x\in\Pol(m+1)+\o(n^\alpha). 
\]
Now, assume 
\[
z\in\Pol(m)+\o(n^\alpha). 
\]
Since \ $\alpha\leq m$ \ we have \ $z_n=\O(n^m)$ \ and, by Lemma \ref{ASL4}, \  $x_n=\O(n^m)$. \ Hence \ $x_{n+k}=\O(n^m)$ \ and from the condition \  $u_n=c+\o(n^{\alpha-m})$ \ we obtain 
\[
z'_n-z_n=(c-u_n)x_{n+k}=
n^\alpha\frac{c-u_n}{n^{\alpha-m}}\frac{x_{n+k}}{n^m}=n^\alpha\o(1)\O(1)=\o(n^\alpha).
\]
Hence the condition \ $z\in\Pol(m)+\o(n^\alpha)$ \ implies 
\[
z'_n=z_n+(z'_n-z_n)\in\Pol(m)+\o(n^\alpha)+\o(n^\alpha)=\Pol(m)+\o(n^\alpha)
\]
and the result follows from the first part of the proof.
\end{proof}

\section{Asymptotically polynomial solutions 1}

In this section, in Theorem \ref{T1}, we obtain our main result. First, in 
Lemma \ref{BHL}, we obtain a certain discrete version of the Bihari's lemma. 
This version is similar to Theorem 1 in \cite{Demidovic} but we do not assume 
the continuity of \ $g$. 

\begin{lemma}\label{BHL}
Assume \ $a,w$ \ are nonnegative sequences, \ $p\in\N$, 
\[
g:[0,\infty)\to[0,\infty), \quad 0\leq\lambda<M, \quad g(\lambda)>0, 
\]
\begin{equation}\label{sum ak}
\sum_{k=0}^\infty a_k\leq\int_{\lambda}^M\frac{dt}{g(t)}, 
\end{equation}
\[
w_n\leq\lambda+\sum_{k=p}^{n-1}a_kg(w_k)
\]
for \ $n\geq p$ \ and \ $g$ \ is nondecreasing. Then \ $w_n\leq M$ \ for \ $n\geq p$. 
\end{lemma}
\begin{proof}   
For \ $n\geq p$, \ let 
\[ 
s_n=\lambda+\sum_{k=p}^{n-1}a_kg(w_k). 
\]
Then, for \ $n\geq p$, \ we have \ $\D s_n=s_{n+1}-s_n=a_ng(w_n)\leq a_ng(s_n)$ \ and  
\[
\int\limits_{s_n}^{s_{n+1}}\frac{dt}{g(t)}\leq
\int\limits_{s_n}^{s_{n+1}}\frac{dt}{g(s_n)}=\frac{\D s_n}{g(s_n)}\leq a_n.
\]
Therefore, using \eqref{sum ak}, we have   
\[
\int\limits_{\lambda}^{s_n}\frac{dt}{g(t)}=\sum_{k=p}^{n-1}
\int\limits_{s_k}^{s_{k+1}}\frac{dt}{g(t)}\leq
\sum_{k=p}^{n-1}a_k\leq\int\limits_{\lambda}^M\frac{dt}{g(t)}.
\]
Since \ $g$ \ is positive on \ $[\lambda,\infty)$, \ we obtain \ $s_n\leq M$. \ Hence 
\[
w_n\leq s_n\leq M 
\]
for \ $n\geq p$. \ The proof is complete. 
\end{proof}
In the proof of Theorem \ref{T1} we also use the following two lemmas.

\begin{lemma}\label{SDL} 
Assume \ $m\in\N(1)$, \ $z\in\SQ$, \ $s\in(-\infty,m-1]$ \ and 
\[
\sum_{n=1}^\infty n^{m-1-s}|\D^mz_n|<\infty.
\]
Then \ $z\in\Pol(m-1)+\o(n^s)$.
\end{lemma}
\begin{proof}
The assertion follows from the proof of Theorem 2.1 in \cite{J Migda 2013}.
\end{proof}

\begin{lemma}\label{BHL2}
If \ $x\in\SQ$ \ and \ $m,n_0\in\N$, \ then there exists \ $L>0$ \ such that
\[
|x_n|\leq n^{m-1}\left(L+\sum_{i=n_0}^{n-1}|\Delta^mx_i|\right) \qquad 
\text{for} \qquad n\geq n_0.
\]
\end{lemma}
\begin{proof}
See \cite[Lemma 7.3]{J Migda 2014}.
\end{proof}
Now we are ready to prove our main result.

\begin{theorem}\label{T1}
Assume \ $m\in\N$, \ $k\in\Z$, \ $c,s,p\in\R$, \ $|c|\neq 1$, \ $s\leq m-1$, \ 
$a,b,u\in\SQ$, 
\[
f:\N\times\R\to\R, \qquad g:[0,\infty)\to[0,\infty), \qquad 
\sigma:\N\to\Z, \qquad \sigma(n)\to\infty,
\]
\[
\sum_{n=1}^\infty n^{m-1-s}|a_n|<\infty, \qquad 
\sum_{n=1}^\infty n^{m-1-s}|b_n|<\infty, \qquad u_n=c+\o(n^{s+1-m}),
\]
$x$ \ is a solution of \eqref{E} and one of the following conditions is satisfied: 
\begin{enumerate}

\item[$(a)$] \ $g$ \ is nondecreasing, \ $f$ \ is \ $(g,m-1)$-bounded, \ 
$\sigma(n)\leq n$ \ for large \ $n$, 
\[
\int_1^\infty\frac{dt}{g(t)}=\infty, 
\]
and \ $x$ \ is \ $(u,k)$-nonoscillatory, 

\item[$(b)$] \ $g$ \ is \smallskip locally bounded, \ $f$ \ is \ $(g,p)$-bounded, \ 
               $x\circ\sigma=\O(n^p)$ \ and the following alternative is satisfied: \ 
               $k(|c|-1)\geq 0$ \ or \ $x\in\O(n^\infty)$ \ or \ 
               $x$ \ is \ $(u,k)$-nonoscillatory,

\item[$(c)$] \ $f$ \ is \smallskip bounded and the following alternative is satisfied: \ 
               $k(|c|-1)\geq 0$ \ or \ $x\in\O(n^\infty)$ \ or \ 
               $x$ \ is \ $(u,k)$-nonoscillatory.
\end{enumerate}
Then 
\begin{equation}\label{aspol}
x\in\Pol(m-1)+\o(n^s).
\end{equation}
\end{theorem}
\begin{proof}
Let \ $z\in\SQ$, 
\[
z_n=x_n+u_nx_{n+k} 
\]
for large \ $n$. \ Assume (b). Since \  $x\circ\sigma=\O(n^p)$ \ and \ $f$ \ is \ $(g,p)$-bounded, we see that the sequence \ $(f(n,x_{\sigma(n)}))$ \ is bounded. Hence 
\[
\sum_{n=1}^\infty n^{m-1-s}|\D^mz_n|<\infty
\]
and, by Lemma \ref{SDL}, we have \ $z\in\Pol(m-1)+\o(n^s)$. \ 
If \ $x$ \ is \ $(u,k)$-nonoscillatory, then 
\[ 
|z_n|=|x_n+u_nx_{n+k}|=|x_n|+|u_nx_{n+k}| 
\]
for large $n$. Hence 
\begin{equation}\label{xnzn}
|x_n|\leq|z_n|
\end{equation}
for large $n$. Therefore \ $x\in\O(n^\infty)$. \ Now, using Lemma \ref{ASL5}, we 
obtain \eqref{aspol}. The proof in the case (c) is analogous. 
\medskip\\ 
Assume (a). There exists an index \ $n_0$ \ such that 
\[ 
|x_n|\leq|z_n|, \qquad \sigma(n)\geq 1, \qquad \sigma(n)\leq n
\]
and \eqref{E} is satisfied for \ $n\geq n_0$. \ 
Choose an index \ $n_1\geq n_0$ \ such that \ $\sigma(n)\geq n_0$ \ for \ $n\geq n_1$. \ 
By Lemma \ref{BHL2}, there exists a positive constant \ $L$ \ such that 
\begin{equation}\label{BHz}
\frac{|z_n|}{n^{m-1}}\leq L+\sum_{j=1}^{n-1}|\Delta^mz_j|
\end{equation}
for any $n$. Let 
\[
L_1=L+\sum_{j=1}^{n_1}|\Delta^mz_j|, \qquad 
L_2=L_1+\sum_{j=1}^\infty|b_j|. 
\]
If \ $n\geq n_1$, \ then, using \eqref{BHz}, \eqref{E}, $(g,m-1)$-boundedness 
of \ $f$, \ and \eqref{xnzn}, we obtain 
\[
\frac{|z_{\sigma(n)}|}{n^{m-1}}\leq\frac{|z_{\sigma(n)}|}{\sigma(n)^{m-1}}\leq
L+\sum_{j=1}^{\sigma(n)-1}|\Delta^mz_j|\leq L+\sum_{j=1}^{n-1}|\D^mz_j|
\]
\[
\leq L_1+\sum_{j=n_1}^{n-1}|\D^mz_j|\leq
L_2+\sum_{j=n_1}^{n-1}|a_j|g\left(\frac{|x_{\sigma(j)}|}{j^{m-1}}\right)\leq
L_2+\sum_{j=n_1}^{n-1}|a_j|g\left(\frac{|z_{\sigma(j)}|}{j^{m-1}}\right).
\]
By Lemma \ref{BHL}, the sequence \ $(z_{\sigma(n)}/n^{m-1})$ \ is bounded. 
Hence, by $(\ref{xnzn})$, 
\[
x\circ\sigma=\O(n^{m-1}). 
\]
Therefore, taking \ $p=m-1$ in (b), we obtain \eqref{aspol}. The proof is complete. 
\end{proof}

\begin{remark}
The condition \ $x\in\O(n^\infty)$ \ is not a consequence of \ 
$x\circ\sigma\in\O(n^\infty)$. \ For example, if \ $x_n=e^n$, \ 
$\sigma(n)=\lfloor\log n\rfloor$ \ $($integer part of \ $\log n$$)$, \ then \ 
$x\circ\sigma=\O(n)$ \ and \ $x\notin\O(n^\infty)$.
\end{remark} 

\begin{remark} 
If the sequence \ $u$ \ is nonnegative, then the class of \ $(u,k)$-nonoscillatory  sequences is larger than the class of nonoscillatory sequences. Moreover, if 
\[
n_1=\min\{n\in\N:\ \sigma(i)\geq 1 \quad \text{for} \quad i\geq n\} 
\]
and we define a full solution of \eqref{E} as a sequence \ $x$ \ such that \eqref{E} 
is satisfied for all \ $n\geq\max(n_1,-k)$, \ then the set of full solutions is 
a subset of the set of all solutions. Hence Theorem \ref{T1} covers the case of full  solutions and, assuming \ $u$ \ is nonnegative, the case of nonoscillatory solutions. 
\end{remark}

\begin{lemma}\label{inf}
If \ $m\in\N(0)$, \ then 
\[
\Pol(m-1)+\o(n^{-\infty})=\bigcap_{k=1}^\infty\left(\Pol(m-1)+\o(n^{-k})\right).
\]
\end{lemma}
\begin{proof}
Let \ $P=\Pol(m-1)$ \ and 
\[ 
x\in\bigcap\limits_{k=1}^\infty\left(P+\o(n^{-k})\right). 
\] 
Then \ $x\in P+\o(1)$ \ and \ $x=\varphi+u$ \ for some \ $\varphi\in P$ \ 
and \ $u\in\o(1)$. \ Since \ $P\cap\o(1)=0$, \ the sequences \ $\varphi$ \ and \ 
$u$ \ are unique. Let \ $k\in\N$. \ Then \ $x\in P+\o(n^{-k})$ \ and by uniqueness 
of \ $u\in\o(1)$ \ we have \ $u\in\o(n^{-k})$. \ Hence \ $u\in\o(n^{-\infty})$ \ and 
we obtain
\[
\bigcap_{k=1}^\infty\left(P+\o(n^{-k})\right)\subset P+\o(n^{-\infty}).
\]
The inverse inclusion is obvious. 
\end{proof}

\begin{corollary} 
Assume all conditions of Theorem \ref{T1} are satisfied and \ $a,b\in\o(n^{-\infty})$. \  Then 
\[
x\in\Pol(m-1)+\o(n^{-\infty}).
\]
\end{corollary}
\begin{proof}
The assertion is a consequence of Theorem \ref{T1} and Lemma \ref{inf}.
\end{proof}

\section{Asymptotically polynomial solutions 2}

In this section, in Theorem \ref{T2}, we obtain a result analogous to Theorem \ref{T1}. 
We replace the spaces of asymptotically polynomial sequences by the spaces of regularly 
asymptotically polynomial sequences. The study of regularly asymptotically polynomial  sequences
\[
\Pol(m)+\D^{-q}\o(1), \qquad q\in\N(0,m)
\]
is motivated by a special case \ $\Pol(m)+\D^{-m}\o(1)$. \ By Remark \ref{lambda}, 
the condition 
\[
z\in\Pol(m)+\D^{-m}\o(1)
\]
is equivalent to the convergence of the sequence \ $\D^mz$ \ and the condition 
\[
\lim_{n\to\infty}\D^mz_n=\lambda
\]
is equivalent to the condition 
\begin{equation}\label{lambda e}
\lim_{n\to\infty}\frac{p!\D^{m-p}z_n}{n^p}=\lambda \quad \text{for any} \quad 
p\in\N(0,m).
\end{equation}
Convergence of the sequence $\D^mz_n$ is comparatively easy to verify and condition  \eqref{lambda e} appears in many papers, see for example \cite{Gleska}, \cite{MJ Migda 2001},  \cite{M Migda 2006}, \cite{Zafer} or the proof of Theorem 3.1 in \cite{Wang Sun}. 

In the next lemma, we establish some basic properties of spaces of regularly 
asymptotically polynomial sequences.

\begin{lemma}\label{APL1}
Assume $m\in\N$, $k\in\N(0,m)$ and $x\in\SQ$. Then
\begin{enumerate}
\item[$(a)$] \ $x\in\D^{-m}\o(1)$ \ $\Longleftrightarrow$ \ 
               $\D^px\in\o(n^{m-p})$ \ for every \ $p\in\N(0,m)$,
\item[$(b)$] \ $x\in \Pol(m)+\D^{-k}\o(1)\ \Longleftrightarrow\                                            \D^px\in\Pol(m-p)+\o(n^{k-p})$ \ for any \ $p\in\N(0,k)$. 
\item[$(c)$] \ $\Pol(m-1)\subset\D^{-m}\o(1)\subset\o(n^m)$, \                                             $\o(n^m)\setminus\D^{-m}\o(1)\neq\emptyset$.
\item[$(d)$] \ $\D^{-m}\o(1)=\left\{z\in\o(n^m):\ \D^pz\in\o(n^{m-p}) \quad 
               \text{for any} \quad p\in\N(0,m)\right\}$.
\end{enumerate}
\end{lemma}
\begin{proof}
(a) If \ $x\in\D^{-m}\o(1)$, \ then \ $\D^mx=\o(1)$ and 
\[
\frac{\D\D^{m-1}x_n}{\D n}=\D^mx_n=\o(1).
\]
By the Stolz-Cesaro theorem $\D^{m-1}x_n=\o(n)$. Hence
\[
\frac{\D\D^{m-2}x_n}{\D n^2}=\frac{n\D\D^{m-2}x_n}{n\D n^2}=
\frac{\D^{m-1}x_n}{n}\frac{n}{\D n^2}\longrightarrow 0.
\]
Again, by the Stolz-Cesaro theorem, \ $\D^{m-2}x_n=\o(n^2)$. \ Analogously \ 
$\D^{m-3}x_n=\o(n^3)$ \ and so on. Inverse implication is obvious.

(b) and (d) are consequences of (a). 

(c) The inclusion 
\[
\Pol(m-1)\subset\D^{-m}\o(1) 
\]
is obvious. The inclusion 
\[ 
\D^{-m}\o(1)\subset\o(n^m) 
\]
is a consequence of (a). If \ $a_n=(-1)^n$, \ then 
\[
\D^ma_n=2^m(-1)^{m+n}\notin\o(1). 
\]
Hence \ $a\in\o(n^m)\setminus\D^{-m}\o(1)$. 
\end{proof}

\begin{remark}
Assume \smallskip \ $m\in\N(0)$, \ $k\in\N(0,m)$. \ If \ $\Pol(m,k)$ \ denotes the 
subspace of \ $\Pol(m)$ \ generated by sequences \ $(n^m),(n^{m-1}),\dots,(n^k)$, \ then 
\[ 
\Pol(m)+\o(n^k)=\Pol(m,k)+\o(n^k) \quad\text{and}\quad 
\Pol(m,k)\cap\o(n^k)=0.
\]
Hence, \ $x\in \Pol(m)+\o(n^k)$ \ if and only if there exist constants 
$c_m,\dots,c_k$ and a sequence $w\in\o(n^k)$ such that 
\[
x_n=c_mn^m+c_{m-1}n^{m-1}+\dots+c_kn^k+w_n.
\]
Moreover, the constants \ $c_m,\dots,c_k$ \ and the sequence \ $w$ \ are unique and 
\[ 
x\in \Pol(m)+\D^{-k}\o(1) \quad \Longleftrightarrow \quad \D^pw_n=\o(n^{k-p}) \quad
\text{for any} \quad  p\in\N(0,k).
\]
If \ $P(m,k)$ \ and \ $D(m,k)$ \ denote the spaces defined by 
\[
P(m,k)=\Pol(m)+\o(n^k) \qquad \text{and} \qquad D(m,k)=\Pol(m)+\D^{-k}\o(1).
\] 
respectively, then we obtain a diagram 
\smallskip\\
\small
\[
\begin{CD}
P(m,0)  @))) P(m,1) @>>> P(m,2) @>>> \dots @>>> P(m,m) @>>> P(m,m+1)\\ 
@AAA       @AAA         @AAA     @.      @AAA         @AAA\\          
D(m,0)  @>>> D(m,1) @>>> D(m,2) @>>> \dots @>>> D(m,m) @))) D(m,m+1)
\end{CD}
\]
\normalsize\smallskip\\
where arrows denote inclusions. Note that 
\[
D(m,0)=\Pol(m)+\o(1)=P(m,0)
\]
and for \ $k>m$ \ we have 
\[
P(m,k)=\o(n^k), \qquad D(m,k)=\D^{-k}\o(1). 
\]
\end{remark}

\begin{remark}\label{lambda}
Assume \ $m\in\N(0)$ \ and \ $x\in\SQ$. \ If 
\begin{equation}
x\in\Pol(m)+\D^{-m}\o(1),
\end{equation}
then, by Lemma \ref{APL1}, the sequence \ $\D^mx$ \ is convergent. On the other hand, 
if \ $\lambda\in\R$ \ and 
\begin{equation}\label{aaa}
\D^mx=\lambda+\o(1),
\end{equation}
then taking \ $w_n=\lambda n^m/m!$ \ we have \ $\D^m(x-w)=\lambda+\o(1)-\lambda=\o(1)$. \ 
Hence
\[
x=w+(x-w)\in\Pol(m)+\D^{-m}\o(1).
\]
Using the Stolz-Cesaro theorem one can show that condition \eqref{aaa} is 
equivalent to the condition 
\[
\lim_{n\to\infty}\frac{p!\D^{m-p}z_n}{n^p}=\lambda \quad \text{for any} \quad 
p\in\N(0,m).
\]
\end{remark}

The next two lemmas are `regular' versions of Lemmas \ref{SDL} and \ref{ASL5}.

\begin{lemma}\label{SDL2}
Let \ $m\in\N$, \ $q\in\N(0,m-1)$, \ $z\in\SQ$ \ and  
\[
\sum_{n=1}^\infty n^{m-q-1}|\D^mz_n|<\infty.   
\]
Then \ $z\in\Pol(m-1)+\D^{-q}\o(1)$.
\end{lemma}
\begin{proof}
By Lemma 2.3 in \cite{J Migda 2013}, there exists \ $w=\o(1)$ \ such that \  $\D^mz=\D^{m-q}w$. \ Choose \ $x\in\SQ$ \ such that \ $\D^qx=w$. \ Then \  $x\in\D^{-q}\o(1)$ \ and 
\[
\D^mz=\D^{m-q}w=\D^{m-q}\D^qx=\D^mx. 
\]
Hence \  $z-x\in\Pol(m-1)$ \ and 
\[
z=z-x+x\in\Pol(m-1)+\D^{-q}\o(1).
\]
\end{proof}

\begin{lemma}\label{ASL7}
Let \ $m\in\N(0)$, \ $q\in\N(0,m)$ \ and \ $u_n=c+\o(n^{-m})$. \ 
Assume \ $k(|c|-1)\geq 0$ \ or \ $x\in \O(n^\infty)$ \ and \ 
$z\in\Pol(m)+\Delta^{-q}\o(1)$. \ Then  
\[
x\in\Pol(m)+\Delta^{-q}\o(1).
\]
\end{lemma}
\begin{proof}
For \ $n\geq n_0$ \ let \ $z'_n=x_n+cx_{n+k}$. \ We will show, by induction on \ $q$, \ 
that 
\[
z'\in\Pol(m)+\Delta^{-q}\o(1) \quad \Longrightarrow \quad x\in\Pol(m)+\Delta^{-q}\o(1).
\]
For \ $q=0$ \ this assertion follows from Lemma \ref{ASL5}. Assume it is true 
for some \ $q\geq 0$. \ Let 
\[
z'\in\Pol(m)+\D^{-(q+1)}\o(1), \qquad z''=\D z', \qquad \text{and} \qquad 
x''=\Delta x. 
\]
Then 
\[
z''_n=\Delta z'_n=\Delta(x_n+cx_{n+k})=
\Delta x_n+c\Delta x_{n+k}=x''_n+cx''_{n+k},
\] 
\[
z''=\Delta z'\in\Delta(\Pol(m)+\Delta^{-(q+1)}\o(1))=\Pol(m-1)+\Delta^{-q}\o(1).
\]
If \ $x\in \O(n^\infty)$, \ then \ $x''=\Delta x\in \O(n^\infty)$. \ By inductive hypothesis 
\[
x''\in \Pol(m-1)+\Delta^{-q}\o(1).
\] 
By equality \ $x''=\Delta x$ \ we obtain \ $x\in\Pol(m)+\Delta^{-(q+1)}\o(1)$. \ 
Now, assume 
\[
z\in \Pol(m)+\Delta^{-q}\o(1). 
\]
Then \ $z_n=\O(n^m)$ \ and, by Lemma \ref{ASL4}, \ $x_n=\O(n^m)$. \ Hence \  $x_{n+k}=\O(n^m)$. \ Since 
\[
u_n=c+\o(n^{-m}), 
\]
we have 
\[
z'_n-z_n=(c-u_n)x_{n+k}=\frac{c-u_n}{n^{-m}}\frac{x_{n+k}}{n^m}=\o(1).
\]
Therefore 
\[
z'=z+(z'-z)\in\Pol(m)+\Delta^{-q}\o(1)+\o(1)=\Pol(m)+\Delta^{-q}\o(1).
\]
Hence, by the first part of the proof, we obtain \ $x\in\Pol(m)+\Delta^{-q}\o(1)$.
\end{proof}

\begin{theorem}\label{T2}
Assume all assumptions of Theorem \ref{T1} are satisfied and moreover let 
\[ 
s=q\in\N(0,m-1] \qquad \text{and} \qquad u_n=c+\o(n^{1-m}). 
\]
Then 
\[
x\in\Pol(m-1)+\D^{-q}\o(1).
\]
\end{theorem}
\begin{proof}
Repeat the proof of Theorem \ref{T1} replacing Lemma \ref{SDL} and 
Lemma \ref{ASL5} by Lemma \ref{SDL2} and Lemma \ref{ASL7}, respectively.
\end{proof}

\end{document}